\documentclass[a4paper,11pt]{amsart}
\usepackage[T1]{fontenc}
\usepackage{textcomp}
\usepackage{lmodern}
\usepackage{a4wide}
\usepackage{enumitem}
\usepackage{mathrsfs}
\usepackage{amsmath, amssymb}
\usepackage{amsthm}

\theoremstyle{plain}
\newtheorem{theorem}{Theorem}[section]
\newtheorem{lemma}[theorem]{Lemma}
\newtheorem{corollary}[theorem]{Corollary}
\newtheorem{prop}[theorem]{Proposition}

\theoremstyle{definition}
\newtheorem{definition}[theorem]{Definition}
\newtheorem{remark}[theorem]{Remark}
\newtheorem{example}[theorem]{Example}

\newtheorem*{lemma*}{Lemma}
\newtheorem*{theorem*}{Theorem}
\newtheorem*{proposition*}{Proposition}
\newtheorem*{corollary*}{Corollary}
\newtheorem*{definition*}{Definition}

\DeclareMathOperator{\SO}{SO}

\newcommand{\C}{\mathbb{C}}
\newcommand{\N}{\mathbb{N}}

\newcommand{\R}{\mathbb{R}}

\newcommand{\Z}{\mathbb{Z}}

\newcommand{\p}{\mathfrak{p}}

\renewcommand{\a}{\mathfrak{a}}

\newcommand{\cfct}{\mathbf{c}}

\usepackage{tikz}
\usetikzlibrary{arrows,shapes,trees}

\newcommand{\what}[1]{\hat{#1}}

\usepackage{mathtools}

\makeatletter
\DeclareRobustCommand\widecheck[1]{{\mathpalette\@widecheck{#1}}}
\def\@widecheck#1#2{%
    \setbox\z@\hbox{\m@th$#1#2$}%
    \setbox\tw@\hbox{\m@th$#1%
       \widehat{%
          \vrule\@width\z@\@height\ht\z@
          \vrule\@height\z@\@width\wd\z@}$}%
    \dp\tw@-\ht\z@
    \@tempdima\ht\z@ \advance\@tempdima2\ht\tw@ \divide\@tempdima\thr@@
    \setbox\tw@\hbox{%
       \raise\@tempdima\hbox{\scalebox{1}[-1]{\lower\@tempdima\box
\tw@}}}%
    {\ooalign{\box\tw@ \cr \box\z@}}}
\makeatother
\begin{document}

\title[Inequalities for the Heckman--Opdam transform]{Hardy--Littlewood inequalities for the Heckman--Opdam transform}
\author{Troels Roussau Johansen}
\address{ 
Mathematisches Seminar \\
Chr.-Albrechts--Universit\"at zu Kiel\\
Ludewig-Mey--Str. 4, DE-24098 Kiel\\
Germany
}

\email{johansen@math.uni-kiel.de}
\begin{abstract}
We establish Hardy--Littlewood inequalities for the Heckman--Opdam transform associated to a general root datum $(\a,\Sigma,m)$ that generalizes an analogous result for the spherical Fourier transform on a Riemannian symmetric space of the non-compact type due to Eguchi and Kumahara. In particular we obtain a more precise Hausdorff--Young inequality that generalizes a recent result due to Narayanan, Pasquale, and Pusti.
\end{abstract}
\maketitle
\section{Introduction}
The classical Hausdorff--Young inequality $\|\what{f}\|_q\leq c_p\|f\|_p$, $1\leq p\leq 2$, $\frac{1}{p}+\frac{1}{q}=1$, for the Euclidean Fourier transform can be viewed as a partial extension of the Plancherel theorem to $L^p$-functions. More generally, the Fourier transform extends to a continuous mapping from $L^p(\R^n)$ into the Lorentz space $L^{p',p}(\R^n)$, a result that is due to Paley. A variation on this theme is provided by the Hardy--Littlewood inequality which may be stated as follows:  Let $f$ be a measurable function on $\R^n$ such that $x\mapsto f(x)\|x\|^{n(1-2/q)}$ belongs to $L^q(\R^n)$, where $q\geq 2$. Then $f$ has a well-defined Fourier transform $\what{f}$ in $L^q(\R^n)$ and there exists a positive constant $A_q$ independent of $f$ such that
\begin{equation}\label{HL-ineq1}
\Bigl(\int_{\R^n}\vert\what{f}(\xi)\vert^qd\xi\Bigr)^{1/q}\leq A_q\Bigl(\int_{\R^n}\vert f(x)\vert^q\|x\|^{n(q-2)}dx\Bigr)^{1/q}.
\end{equation}

By duality and general properties of the Fourier transform, one has the following equivalent formulation: For every $p\in(1,2)$ there exists a positive constant $B_p$ independent of $f$ such that
\begin{equation}\label{HL-ineq2}
\Bigl(\int_{\R^n}|\what{f}(\xi)|^p|\xi|^{n(p-2)}\,d\xi\Bigr)^{1/p}\leq B_p\Bigl(\int_{\R^n}|f(x)|^p\,dx\Bigr)^{1/p}.
\end{equation}

An analogue of \eqref{HL-ineq1} for the spherical transform on a Riemannian symmetric space $G/K$ was obtained by Eguchi and Kumahara in \cite[Theorem~1, Section~5]{Eguchi-Kumahara}:
\begin{theorem}\label{thm.1}
	Let $q\geq 2$. The spherical Fourier transform can be defined for $K$-invariant functions $f$ on $G/K$ with the property that $f\cdot\sigma^{n(1-2/q)}\Omega^{1-2/q}$ belongs to $L^q(K\setminus G/K)$, and there exists a positive constant $A_q$ that is independent of $f$ such that
	\begin{equation}\label{ineq-HL-GK}
	\Bigl(\frac{1}{\vert W\vert}\int_{\a^*}\vert\widetilde{f}(\lambda)\vert^q\,\vert\mathbf{c}(\lambda)\vert^{-2}\,d\lambda\Bigr)^{1/q} \leq A_q\Bigl(\int_G\vert f(x)\vert^q\sigma(x)^{n(q-2)}\Omega(x)^{q-2}\,dx\Bigr)^{1/q}
	\end{equation}
	for all $f\in\mathcal{S}(K\setminus G/K)$.
\end{theorem}
Here $\sigma(x)=\left<X,X\right>^{1/2}$ where $\left<\cdot,\cdot\right>$ is the Cartan--Killing form and $G\ni x=k\exp X\in K\times\mathfrak{p}$, and $\Omega(\exp H)=c\prod_{\alpha\in\Sigma}\vert\sinh\alpha(H)\vert^{m(\alpha)}$, $H\in\mathfrak{a}$, the usual weight and $\mathcal{S}(K\setminus G/K)$ an $L^2$-based Schwartz space of $K$-invariant functions on $G/K$. An interpolation argument leads to an analogous statement for exponents below $2$:

\begin{theorem}\label{thm.2}
Let $p\in(1,2]$ and $\frac{1}{p}+\frac{1}{q}=1$. Let $r\in[p,q]$ and set $\mu=\frac{1}{r}+\frac{1}{q}-1=\frac{1}{r}-\frac{1}{p}$. Then there exists a positive constant $B_r$ independent of $f$ such that
\[
\Bigl(\frac{1}{|W|}\int_{\a^*}\vert\widetilde{f}(\lambda)\vert^q\vert\mathbf{c}(\lambda)\vert^{-2}\,d\lambda\Bigr)^{1/q} \leq B_r\Bigl(\int_G\vert f(x)\vert^r\sigma(x)^{-n\mu r}\Omega(x)^{-\mu r}\,dx\Bigr)^{1/r}\]
for all $f$ satisfying $f\cdot\sigma^{-n\mu}\Omega^{-\mu}\in L^r(K\setminus G/K)$.
\end{theorem}

 It was remarked in the MathSciNet review by Michael Cowling that one could simplify the proof of Eguchi and Kumahara by means of more refined interpolation techniques. These  were later incorporated in \cite{Mohanty-II} where the authors established an analogue of \eqref{HL-ineq2} for the Helgason--Fourier transform on a noncompact Riemannian symmetric space of rank one: It holds that
\begin{equation}\label{eqn.MRSS}
\int_{\a^*} \|\widetilde{f}(\lambda,\cdot)\|_{L^1(K)}|\lambda|^{p-2}(1+|\lambda|)^{-(m_\gamma+m_{2\gamma})}|\mathbf{c}(\lambda)|^{-2}\,d\lambda\leq C\|f\|_p^p\end{equation}
for $1<p<2$. According to \cite[Remark~4.6]{Mohanty-II}, their method also works for higher rank spaces. While we share this sentiment, it turns out to be slightly involved to fill in the necessary details. One may also object that the appearance of the average over $K$ is not natural. A different version was recently obtained in \cite{Ray-Sarkar-trans}, to which we shall return later. A further drawback of \eqref{eqn.MRSS} is that for $p=2$ it does not resemble the Parseval identity, and section \ref{sec.HY-ineqs} opens with the observation that the analogue of \eqref{eqn.MRSS} for the Heckman--Opdam transform, or even just the Jacobi transform in rank one, does not hold for arbitrary non-negative root multiplicities. We also wish to emphasize a quantitative difference between \eqref{HL-ineq1} and \eqref{HL-ineq2}: In the first inequality a weight is introduced on the function-side, whereas the second inequality incorporates a weight on the Fourier transform side. Theorem \ref{thm.1} and theorem \ref{thm.2} therefore resemble \eqref{HL-ineq1}, whereas \eqref{eqn.MRSS} resembles  \eqref{HL-ineq2}.
\smallskip

It is the purpose of the present paper to obtain analogues of \eqref{HL-ineq1} and \eqref{HL-ineq2} for the Heckman--Opdam transform associated to a triple $(\a,\Sigma,m)$, where $\a$ is an Euclidean $n$-dimensional vector space, $\Sigma$ a root system in $\a^*$ and $m$ a positive multiplicity function.  In order to place the contributions of the present paper in perspective, the reader is reminded that some classical aspects of the $L^2$-theory for hypergeometric Fourier analysis in root systems (that is, Plancherel and Paley--Wiener theorems and an inversion formula) was already obtained in \cite{Opdam-acta}, whereas the $L^p$-analysis is  much more recent. As far as we can ascertain, the first decisive contribution was given in the recent publication \cite{Narayanan-Pasquale-Pusti}, and the results we obtain should be seen as natural contributions to the general theme of classical harmonic analysis in a root system framework,
\medskip

The details pertaining to harmonic analysis in root systems will be presented in section \ref{sec.root1}. There are several standard references but we follow closely the presentation in \cite{Narayanan-Pasquale-Pusti} as far as the Heckman--Opdam theory is concerned. Section \ref{sec.root1} also summarizes the interpolation theorems for Lorentz spaces. An immediate consequence is a generalized Hausdorff--Young inequality of Paley-type. Section \ref{sec.HY-ineqs} presents several versions of the Hardy--Littlewood inequality for the Heckman--Opdam transforms, corresponding to different weights. The last section briefly outlines a generalization of the Eguchi--Kumahara result for the Cartan motion groups. On can introduce a  `flat' Heckman--Opdam transform $\mathcal{F}_0$ in analogy with generalized Bessel transform on the flat space $G_0/K$, and we obtain Hardy--Littlewood inequalities for $\mathcal{F}_0$ as well. This involves generalized Bessel-type functions associated with root systems that were already considered by Opdam in \cite[Section~6]{Opdam-Bessel}. The connection to spherical functions on the Cartan motion group was explicitly indicated in \cite[Remark~6.12]{Opdam-Bessel} and later established formally in \cite{deJeu-PW}.

\section{Harmonic analysis in root systems}\label{sec.root1}
Let $\a$ be an $n$-dimensional real Euclidean vector space with inner product $\langle\cdot,\cdot\rangle$ and let $\a^*$ denote the linear dual of $\a$. For $\lambda\in\a^*$ let $x_\lambda$ be the unique vector in $\a$ such that $\lambda(x)=\langle x,x_\lambda\rangle$ for every $x\in\a$. Define an inner product  on $\a^*$ via $\langle \lambda,\mu\rangle=\langle x_\lambda,x_\mu\rangle$, and let $\a_\C$ and $\a_\C^*$ denote the complexifications of $\a$ and $\a^*$. The inner products on $\a$ and $\a^*$ extend by $\C$-linearity to inner products on $\a_\C$ and $\a_\C^*$ that will denoted by the same symbol. Set $\lambda_\alpha=\frac{\langle\lambda,\alpha\rangle}{\langle\alpha,\alpha\rangle}$ and $|x|=\langle x,x\rangle^{1/2}$ for $x\in\a$. 

Let $\Sigma$ be a root system in $\a^*$ and let $W$ denote the associated Weyl group generated  by the root reflections $r_\alpha:\lambda\mapsto\lambda -2\lambda_\alpha\alpha$ for $\alpha\in\Sigma$. Fix a compatible set $\Sigma^+$ of positive roots in $\Sigma$ and let $\Pi=
\{\alpha_1,\ldots,\alpha_n\}\subset\Sigma^+$ be the associated set of simple roots. Let $\Sigma_0$ denote the set of roots in $\Sigma$ that are indivisible, in the sense that if $\alpha$ belongs to $\Sigma_0$, then $\alpha/2$ is not a root. A (strictly) positive multiplicity function is a $W$-invariant function $m:\Sigma\to(0,\infty)$. We often write $m_\alpha=m(\alpha)$. By $W$-invariance it holds that $m_{w\alpha}=m_\alpha$ for all $\alpha\in\Sigma$ and $w\in W$. We adhere to the conventions in \cite{Narayanan-Pasquale-Pusti} rather than Heckman and Opdam: Their root system $\mathfrak{R}$ and multiplicity function $k$ are related to $\Sigma$ and $m$ above by the identities $\mathfrak{R}=\{2\alpha\,:\,\alpha\in\Sigma\}$, $k_{2\alpha}=m_\alpha/2$ for $\alpha\in\Sigma$. Set $\Sigma_0^+=\Sigma^+\cap\Sigma_0$. The complexification $\a_\C$ of $\a$ may be viewed as the Lie algebra of the complex torus $A_\C=\a_\C/\{2\pi ix_\alpha/\langle\alpha,\alpha\rangle\,:\,\alpha\in\Sigma\}\Z$. Let $\exp:\a_\C\to A_\C$ be the exponential map. The real form $A=\exp\a$ of $A_\C$ is an abelian subgroup of $A_\C$ with Lie algebra $\a$ such that $\exp:\a\to A$ is a diffeomorphism, by means of which we shall often identify $\a$ with $A$. The $W$-action extends to $\a$ by duality, and to $\a_\C^*$ and $\a_\C$ by $\C$-linearity. Moreover $W$ acts via the left regular representation of functions on either one of these. The positive Weyl chamber $\a^+$ is defined as the set of elements $x\in\a$ for which $\alpha(x)>0$ for all $\alpha\in\Sigma^+$.

Let $P=\{\lambda\in\a^*\,:\,\lambda_\alpha\in\Z\text{ for all }\alpha\in\Sigma\}$ denote the  restricted weight lattice. To $\lambda\in P$ is associated the single-valued exponential $e^\lambda:A_\C\to\C$ given by $e^\lambda(h)=e^{\lambda(\log h)}$, and these are the characters of $A_\C$. It can be seen that $\mathrm{span}_\C\{e^\lambda\}$ is isomorphic to the ring $\C[A_\C]$ of regular functions on $A_\C$, the latter viewed as an algebraic variety, and $W$ acts on it by $w(e^\lambda):=e^{w\lambda}$ (which is well-defined since $P$ is $W$-invariant). The set of regular points for the $W$-action on $A_\C$ coincides with the set $A_\C^{\mathrm{reg}}=\{h\in\C\,:\,e^{2\alpha(\log h)}\neq 1\text{ for all }\alpha\in\Sigma\}$, and the  algebra $\C[A_\C^{\mathrm{reg}}]$ of regular functions on $A_\C^{\mathrm{reg}}$ is the subalgebra of the quotient field of $\C[A_\C]$ generated by $\C[A_\C]$ and $\{\frac{1}{1-e^{-2\alpha}}\,:\,\alpha\in\Sigma^+\}$. We denote by $\C[A_\C^{\mathrm{reg}}]^W$ the subalgebra of $W$-invariant elements.

Let $S(\a_\C)$ denote the symmetric algebra over $\a_\C$ and let $S(\a_\C)^W$ denote the subalgebra of its $W$-invariant elements. An element $p\in S(\a_\C)$ gives rise to a constant coefficient differential operator $\partial(p)$ acting on functions $f$ on $A_\C$ such that $\partial(x)$ is the directional derivative in the direction of $x$ for every $x\in\a$. We shall denote the algebra $\{\partial(p)\,:\,p\in S(\a_\C)\}$ by the symbol $S(\a_\C)$, which is justified since $p\mapsto\partial(p)$ is an algebra isomorphism. Let $\mathbb{D}(A_\C^{\mathrm{reg}})=\C[A_\C^{\mathrm{reg}}]\otimes S(\a_\C)$ denote the algebra of differential operators on $A_\C$ with coefficients in $\C[A_\C^{\mathrm{reg}}]$ and $\mathbb{D}(A_\C^{\mathrm{reg}})^W$ its subalgebra of $W$-invariant elements, where $W$ acts by $w(\varphi\otimes\partial(p))=(w\varphi)\otimes\partial(wp)$. 

We introduce an associative algebra structure on $\mathbb{D}(A_\C^{\mathrm{reg}})\otimes\C[W]$ via $(D_1\otimes w_1)\cdot(D_2\otimes w_2)=D_1w_1(D_2)\otimes w_1w_2$, where $(wD)(wf):=w(Df)$. Elements of $\mathbb{D}(A_\C^{\mathrm{reg}})\otimes\C[W]$ are the differential-reflection operators on $A_\C^{\mathrm{reg}}$, and they act on functions $f$ on $A_\C^{\mathrm{reg}}$ by $(D\otimes w)f=D(wf)$.
\begin{definition}
The \emph{Cherednik operator} $T_x\in\mathbb{D}(A_\C^{\mathrm{reg}})\otimes\C[W]$ associated with $x\in\a$ is defined by 
\[T_x=\partial_x-\rho(x)+\sum_{\alpha\in\Sigma^+}m_\alpha\alpha(x)(1-e^{-2\alpha})^{-1}\otimes(1-r_\alpha),\]
where $2\rho=\sum_{\alpha\in\Sigma^+}m_\alpha\alpha\in\a^*$.
\end{definition}
It is a deep result that the operators $T_x$ commute, an important consequence of which is that the map $x\mapsto T_x$ extends uniquely to an algebra homomorphism $p\mapsto T_p$ of $S(\a_\C)$ into $\mathbb{D}(A_\C^{\mathrm{reg}})\otimes\C[W]$. 

Define $\Upsilon:\mathbb{D}(A_\C^{\mathrm{reg}})\otimes\C[W]\to\mathbb{D}(A_\C^{\mathrm{reg}})$ by $\Upsilon(\sum_j D_j\otimes w_j)=\sum_jD_j$. Then $\Upsilon(P)f=Pf$ for every $P\in\mathbb{D}(A_\C^{\mathrm{reg}})\otimes\C[W]$ and every $W$-invariant function $f$ on $A_\C^{\mathrm{reg}}$. In particular we can define $D_p:=\Upsilon(T_p)$ for every $p\in S(\a_\C)$.  One can show that $D_p$ belongs to $\mathbb{D}(A_\C^{\mathrm{reg}})^W$ whenever $p\in S(\a_\C)^W$, and that $\mathbb{D}:=\{D_p\,:\,p\in S(\a_\C)^W\}$ is a commutative subalgebra of $\mathbb{D}(A_\C^{\mathrm{reg}})^W$, see \cite[p.~232]{Narayanan-Pasquale-Pusti} for details. In the geometric case where $(\a,\Sigma,m)$ corresponds to the root datum of a Riemannian symmetric space, $\mathbb{D}$ is the algebra of radial components along $A$ of $G$-invariant differential operators on $G/K$. The analogue of the radial component of the Laplace--Beltrami operator is the operator $D_{p_L}$, where $p_L\in S(\a_\C)^W$ is the polynomial $p_L(\lambda)=\langle\lambda,\lambda\rangle$. Then $D_{p_L}=L+\langle\rho,\rho\rangle$ where 
\[L=L_\a+\sum_{\alpha\in\Sigma^+}m_\alpha\coth \alpha \partial(x_\alpha);\]
here $L_\a$ is the Euclidean Laplace operator on $\a$, and $\coth \alpha=\frac{1+e^{-2\alpha}}{1-e^{-2\alpha}}$.

\begin{definition}
Let $\lambda\in\a_\C^*$ be fixed. The \emph{hypergeometric function with spectral parameter} $\lambda\in\a_\C^*$ is the unique analytic $W$-invariant function $\varphi_\lambda$ on $\a$ that satisfies the system of differential equations
\[D_p\varphi = p(\lambda)\varphi,\quad p\in S(\a_\C)^W\]
and is normalized by $\varphi_\lambda(0)=1$.
\end{definition}
\begin{example}[The rank one case]
In the case $n=1$, $\Sigma^+$ consists of at most two elements, $\alpha$ and $2\alpha$. Identify $\a$ and $\a^*$ with $\R$ by setting $x_\alpha/2\equiv 1$ and $\alpha\equiv 1$. Then $\a_+\simeq (0,\infty)$, and $W=\{-1,1\}$ acts on $\R$ and $\C$ by multiplication. The algebra $\mathbb{D}$ is generated by a single element, for example the operator $D_{\rho_L}=L+\rho^2$, where $\rho=m_\alpha/2+m_{2\alpha}$. The hypergeometric system of differential equations used to define $\varphi$ reduces to the sdifferential equation
\[\frac{d^2\varphi}{dz^2}+(m_\alpha\coth z+m_{2\alpha}\coth(2z))\frac{d\varphi}{dz}=(\lambda^2-\rho^2)\varphi\]
which may be transformed into the hypergeometric differential equation
\[\zeta(1-\zeta)\frac{d^2\psi}{d\zeta^2}+(c-(1+a+b)\zeta)\frac{d\psi}{d\zeta}-ab\zeta=0\]
where $\zeta=\frac{1}{2}(1-\cosh z)$, $a=\frac{\lambda+\rho}{2}$, $b=\frac{-\lambda+\rho}{2}$, and $c=\frac{m_\alpha+m_{2\alpha}+1}{2}$. The solution $\varphi_\lambda$ is therefore  the Jacobi functions 
\[\varphi_\lambda(t)={_2}F_1\Bigl(\frac{m_\alpha/2+m_{2\alpha}+\lambda}{2},\frac{m_\alpha/2+m_{2\alpha}-\lambda}{2};\frac{m_\alpha+m_{2\alpha}+1}{2};-\sinh^2t\Bigr)\]
which are well known to describe the elementary spherical functions on a rank one Riemannian symmetric space $G/K$.
\end{example}

Existence, uniqueness and regularity properties of $\varphi_\lambda$ were investigated in several publications of Heckman and Opdam, later sharpened by Schapira \cite{Schapira} (where the functions are denoted $F_\lambda$ and their non-symmetric versions $G_\lambda$) and most recently by Narayanan, Pasquale, and Pusti \cite{Narayanan-Pasquale-Pusti}. Since we do not need to estimate the functions $\varphi_\lambda$ and the associated Harish-Chandra series expansions in the results that follow, we merely refer the reader to \cite[Sections~2--4]{Narayanan-Pasquale-Pusti} for the details.

\begin{definition}
The \emph{Heckman--Opdam} transform of a function $f\in C_c^\infty(\a)^W$ is defined by
\[\mathcal{F}f(\lambda)=\int_\a f(x)\varphi_\lambda(x)\,d\mu(x).\]
\end{definition}
Often $\mathcal{F}$ is called the hypergeometric Fourier transform, since the functions $\varphi_\lambda$ can be seen as generalized hypergeometric functions. In some literature, such as \cite{Schapira}, these are denoted by $F_\lambda$. Their non-symmetric version appear as $G_\lambda$, in terms of which one can define a hypergeometric Fourier transform of a function $f\in C_c^\infty(\a)$ that is not necessarily $W$-invariant. The terminology `hypergeometric' was certainly used by Delorme in \cite{Delorme-hyper} but the transform itself was studied earlier by Cherednik from a different perspective, and by Opdam in \cite{Opdam-acta}. When $\mathcal{F}$ acts on functions that might not be $W$-invariant, we therefore talk about the Cherednik--Opdam transform or the hypergeometric transform. Its restriction to $W$-invariant functions is the transform that Heckman and Opdam studied, hence the name. A convenient analogy is to think of the Heckman--Opdam transform as a root space generalization of the spherical Fourier transform associated with a Riemannian symmetric space. The more general Cherednik--Opdam transform is not related to the Helgason--Fourier transform on a symmetric space, however.
\begin{definition}
	The $\cfct$-function associated with $(\a,\Sigma,m)$ is defined by
	\[\cfct(\lambda)=c\prod_{\alpha\in\Sigma_0^+}\cfct_\alpha(\lambda),\quad \text{with}\quad \cfct_\alpha(\lambda)=\frac{2^{-\lambda_\alpha}}{\Gamma(\frac{\lambda_\alpha}{2}+\frac{m_\alpha}{4}+\frac{1}{2})\Gamma(\frac{\lambda_\alpha}{2}+\frac{m_\alpha}{4}+\frac{m_{2\alpha}}{2})},\]
	where $c$ is a normalizing constant that is chosen so that $\cfct(\rho)=1$.
\end{definition}
It is known that $|\cfct_\alpha(\lambda)|^{-2}\asymp|\!\left<\lambda,\alpha\right>\!|(1+\|\lambda\|)^{m_\alpha+m_{2\alpha}-2}$, so by the Cauchy--Schwarz inequality in $\a$,
\begin{equation}\label{eqn.c-est1}
|\cfct_\alpha(\lambda)|^{-2}\asymp \begin{cases} \|\lambda\|^{2}(1+\|\lambda\|)^{m_\alpha+m_{2\alpha}-2}&\text{for }\lambda\in i\a^*\text{ with } \|\lambda\| \text{ large},\\
\|\lambda\|^{2}&\text{for } \lambda\in i\a^*\text{ with } \|\lambda\|\lesssim 1\end{cases}.
\end{equation}
In particular,
\[|\cfct(\lambda)|^{-2}\asymp \prod_{\alpha\in\Sigma_0^+}|\!\left<\lambda,\alpha\right>\!|^2(1+|\!\left<\lambda,\alpha\right>\!|)^{m_\alpha+m_{2\alpha}-2} \asymp \|\lambda\|^{2|\Sigma_0^+|}(1+\|\lambda\|)^{\beta-2|\Sigma_0^+|},\]
where $\beta=\sum_{\alpha\in\Sigma_0^+}(m_\alpha+m_{2\alpha})$ and where $|\Sigma_0^+|$ is the cardinality of $\Sigma_0^+$. 

Let $dx$ denote a fixed normalization of the Haar measure on the abelian group $\a$, and associate to $(\a,\Sigma,m)$ the weighted measure $d\mu(x)=J(x)\,dx$ on $\a$, where
\[J(x)=\prod_{\alpha\in\Sigma^+}|e^{\alpha(x)}-e^{-\alpha(x)}|^{m_\alpha}.\]

It is known, cf.  \cite[Theorem~1.13]{Narayanan-Pasquale-Pusti} and the references to the literature, that there exists a suitable normalization of the measure $d\lambda$ on $i\a^*$ such that the transform $\mathcal{F}$ extends to an isometric isomorphism from $L^2(\a,d\mu)^W$ onto $L^2(i\a^*,|\cfct(\lambda)|^{-2}d\lambda)^W$. Moveover, for $f\in C_c^\infty(\a)^W$,
\[f(x)=\int_{i\a^*}\mathcal{F}f(\lambda)\varphi_{-\lambda}(x)|\cfct(\lambda)|^{-2}\,d\lambda\]
for all $x\in\a$.

\begin{definition}
For $p\in(0,2]$, set $\epsilon_p=\frac{2}{p}-1$. Let $C(\epsilon_p)$ be the convex hull in $\a^*$ of the set $\{\epsilon_p w\rho\,:\,w\in W\}$, and let $\a^*_{\epsilon_p}=C(\epsilon_p\rho)+i\a^*$.
\end{definition}

The following two versions of the Hausdorff--Young theorem for the Heckman--Opdam transform were recently established by Narayanan, Pasquale, and Pusti, cf. \cite[Lemma~5.2, Lemma~5.3]{Narayanan-Pasquale-Pusti}. The first proof involves Riesz--Thorin interpolation, whereas the second uses interpolation with an analytic family of operators.
\begin{lemma} 
Let $f\in L^p(A,d\mu)^W$. Then the following properties hold.
\begin{enumerate}[label=(\alph*)]
\item The hypergeometric transform $\mathcal{F}f(\lambda)$ is well defined for all $\lambda$ in the interior of $\a^*_{\epsilon_p}$ and defines a holomorphic function.
\item 	Let $p,q$ be so that $1<p<2$ and $\frac{1}{p}+\frac{1}{q}=1$. Then there exists a positive constant $c_p$ independent of $f$ so that 
\begin{equation}\label{ineq-HY-root-1}
\|\mathcal{F}f\|_q = \Bigl(\int_{i\a^*}|\mathcal{F}f(\lambda)|^q\,d\nu(\lambda)\Bigr)^{1/q}\leq c_p\|f\|_p.
\end{equation}
\end{enumerate}
\end{lemma}
A more precise formulation is given as follows. Let $f\in L^1(\a,d\mu)^W\cap L^2(\a,d\mu)^W$. If $1\leq p\leq 2$, $q=\frac{p}{p-1}$, then $\|\mathcal{F}f\|_{q}\leq c_p\|f\|_p$. Since $L^p$ is dense in $L^1\cap L^2$ for $1\leq p\leq 2$, the Heckman--Opdam transform $\mathscr{F}_pf$ can be defined uniquely for all $f\in L^p(\a,d\mu)^W$, $1\leq p\leq 2$, so that $\mathscr{F}_p:L^p(\a,d\mu)^W\to L^q(i\a^*,d\nu)^W$ is a linear contraction with $\mathscr{F}_pf=\mathcal{F}f$ for all $f\in L^1(\a,d\mu)^W\cap L^2(\a,d\mu)^W$. It is known from \cite[Theorem~5.4]{Narayanan-Pasquale-Pusti} that $\mathscr{F}_p$ is injective on $L^p(\a,d\mu)^W$ whenever $p\in[1,2]$. 

\begin{lemma}\label{lemma.NPP-2}
Let $f\in L^p(\a,d\mu)^W$ for some $p\in(1,2)$ and let $\eta$ be in the interior of $C(\epsilon_p\rho)$. Then the following properties hold:
\begin{enumerate}[label=(\alph*)]
\item Assume $\frac{1}{p}+\frac{1}{q}=1$. There exists a positive constant $C_{p,\eta}$ such that for all $f\in L^p(\a,d\mu)^W$,
\begin{equation}\label{ineq-HY-root-2}
\Bigl(\int_{i\a^*}|\mathcal{F}f(\lambda+\eta)|^q|\cfct(\lambda)|^{-2}\,d\lambda\Bigr)^{1/q}\leq C_{p,\eta}\|f\|_p.
\end{equation}
\item It holds that $\sup_{\lambda\in i\a^*}|\mathcal{F}f(\lambda+\eta)|\leq C_{p,\eta}\|f\|_p$ and 
\[\lim_{\lambda\in\a^*_{\epsilon_p},|\Im\lambda|\to\infty} |\mathcal{F}f(\lambda)|=0.\]
\end{enumerate}
\end{lemma}
An important ingredient in the proof of both results is a characterization of the set of spectral parameters $\lambda$ for which $\varphi_\lambda$ is bounded. This description was obtained in  \cite[Theorem~4.2]{Narayanan-Pasquale-Pusti}: $\varphi_\lambda$ is bounded if and only if $\lambda\in C(\rho)+i\a^*$, in which case $|\varphi_\lambda(x)|\leq 1$ for every $x\in\a$. Note that $C(\epsilon_p\rho)\subset C(\rho)$ for $p\geq 2$. Also note that $C(\epsilon_2\rho)=\{0\}$.

\begin{lemma}\label{lemma.p1p2}
\begin{enumerate}[label=(\roman*)]
\item If $f$ belongs to $(L^{p_1}\cap L^{p_2})(\a,d\mu)^W$ for some $p_1,p_2\in[1,2]$, then $\mathscr{F}_{p_1}f=\mathscr{F}_{p_2}f$ $\nu$-almost everywhere on $i\a^*$.
\item If $h$ belongs to $(L^{q_1}\cap L^{q_2})(i\a^*,d\nu)^W$ for some $q_1,q_2\in[1,2]$, then $\mathscr{I}_{q_1}h=\mathscr{I}_{q_2}h$ $\mu$-almost everywhere on $\a$.
\end{enumerate}
\end{lemma}
\begin{proof}
\begin{enumerate}[label=(\roman*)]
\item Choose a sequence $\{g_n\}_{n=1}^\infty$ of simple $W$-invariant functions on $\a$ such that 
\[\lim_{n\to\infty}\|f-g_n\|_{p_1}=\lim_{n\to\infty}\|f-g_n\|_{p_2}=0.\]
Each function $\mathcal{F}g_n$ belongs to $(L^{p_1^\prime}\cap L^{p_2^\prime})(i\a^*,d\nu)^W$ by the Hausdorff--Young inequality \eqref{ineq-HY-root-1}, and
\[\lim_{n\to\infty}\|\mathscr{F}_{p_1}f-\mathcal{F}g_n\|_{p_1^\prime}=\lim_{n\to\infty}\|\mathscr{F}_{p_2}f\mathcal{F}g_n\|_{p_2^\prime}=0.\]
One can therefore extract subsequences $\{\mathcal{F}g_{n_k}\}_{k=1}^\infty$ and $\{\mathcal{F}g_{n_l}\}_{l=1}^\infty$ of $\{\mathcal{F}g_n\}_{n=1}^\infty$ such that $\mathcal{F}g_{n_k}\to\mathscr{F}_{p_1}f$ and $\mathcal{F}g_{n_l}\to\mathscr{F}_{p_2}f$ $\nu$-almost everywhere on $i\a^*$, from which it follows that $\mathscr{F}_{p_1}f=\mathscr{F}_{p_2}f$ $\nu$-almost everywhere on $i\a^*$ as claimed.
\item Since $\mathscr{I}_{q_j}h\in L^{q_j^\prime}(\a,d\mu)^W$ for $j=1,2$, the claim follows from the injectivity of $\mathscr{F}_p$ for $p\in(1,2]$ and (i).
\end{enumerate}
\end{proof}
The remainder of the section is concerned with interpolation results in Lorentz spaces that will be needed in our proof of the Hardy--Littlewood inequality. The interested reader may consult \cite[Chapter~V]{Stein-Weiss-analysis} for detailed proofs and historical remarks. Let $(X,\mu)$ be a $\sigma$-finite measure space and let $p\in(1,\infty)$. Define
\[\|f\|^*_{p,q}=\begin{cases}\displaystyle \Bigl(\frac{q}{p}\int_0^\infty t^{q/p-1}f^*(t)^q\,dt\Bigr)^{1/q}&\text{if } q<\infty\\ 
\displaystyle \sup_{t>0} t\lambda_f(t)^{1/p}&\text{when } q=\infty\end{cases}\]
where $\lambda_f$ is the distribution function of $f$ and $f^*$ the non-increasing rearrangement of $f$, that is
\[\lambda_f(s) =\mu(\{x\in X\,:\, |f(x)|>s\}) \quad\text{and}\quad f^*(t)=\inf \{s\,:\,\lambda_f(s)\leq t\}.\]
By definition, the Lorentz space $L^{p,q}(X)$ consists of measurable functions $f$ on $X$ for which $\|f\|^*_{p,q}<\infty$. 

\begin{definition}
Let $(X,d\mu)$ and $(Y,d\overline{\nu})$ be $\sigma$-finite measure spaces. A linear operator $T:L^p(X,d\mu)\to L^q(Y,d\overline{\nu})$ is \emph{strong type $(p,q)$} if it is continuous on $L^p(X,d\mu)$. Moreover, $T$ is \emph{weak type $(p,q)$} if there exists a positive constant $K$ independent of $f$ such that for all $f\in L^p(X,d\mu)$ and all $t>0$,
\[\mu\bigl(\bigl\{y\in Y\,:\, |Tf(y)|>t\bigr\}\bigr)\leq \Bigl(\frac{K}{s}\|f\|_{L^p(X,d\mu)}\Bigr)^q.\]
The infimum if such $K$ is the weak type $(p,q)$ norm of $T$.
\end{definition}

For $1\leq p,q\leq\infty$, $L^{p,p}(X,d\mu)=L^p(X,d\mu)$, and $\|f\|^*_{p,q_2}\leq\|f\|^*_{p,q_1}$ whenever $q_1\leq q_2$. The following version of H\"older's inequality for Lorentz spaces is well-known.
\begin{prop}
Let $0<p,q,r\leq\infty$, $0<s_1,s_2\leq\infty$. Then
\[\|f\cdot g\|_{r,s}^* \leq C_{p,q,s_1,s_2}\|f\|^*_{p,s_1}\|g\|^*_{q,s_2}\]
where $\frac{1}{p}+\frac{1}{q}=\frac{1}{r}$ and $\frac{1}{s_1}+\frac{1}{s_2}=\frac{1}{s}$. 
\end{prop}
The dual of $L^{p,q}(X,d\mu)$ is the space $L^{p^\prime,q^\prime}(X,d\mu)$, where $\frac{1}{p}+\frac{1}{p^\prime}=1=\frac{1}{q}+\frac{1}{q^prime}$, and the dual of $L^{1,q}(X,d\mu)$ is $\{0\}$ when $1<q<\infty$. The following interpolation theorem is classical and can be found as Theorem 3.15 in \cite[Chapter~V]{Stein-Weiss-analysis}. It subsumes the interpolation theorem of Marcinkiewicz, for example.

\begin{theorem}[Interpolation between Lorentz spaces]\label{thm.lorentz-inter} Suppose $T$ is a subadditive operator of (restricted) weak types $(r_j,p_j)$, $j=0,1$, with $r_0<r_1$ and $p_0\neq p_1$, then there exists a constant $B=B_\theta$ such that  $\|T\|^*_{p,q}\leq B\|f\|^*_{r,q}$ for all $f$ belonging to the domain of $T$ and to $L^{r,q}$, where $1\leq q\leq\infty$, 
\begin{equation}\label{eqn.interpolate-indices}
\frac{1}{p}=\frac{1-\theta}{p_0}+\frac{\theta}{p_1},\quad \frac{1}{r}=\frac{1-\theta}{r_0}+\frac{\theta}{r_1}\quad\text{and}\quad 0<\theta<1.
\end{equation}
\end{theorem}	
\begin{corollary}[Paley's extension of the Hausdorff--Young inequality]\label{cor.HY-Lorentz}
	If $f\in L^p(\R^n)$, $1<p\leq 2$, then its Fourier transform $\what{f}$ belongs to $L^{p',p}(\R^n)$ and there exists a constant $B=B_p$ independent of $f$ such that $\|\what{f}\|_{p',p}^*\leq B_p\|f\|_p$, where $\frac{1}{p}+\frac{1}{p'}=1$. In particular the Fourier transform is a continuous linear mapping from $L^p(\R^n)$ to the Lorentz space $L^{p',p}(\R^n)$ for $1<p<2$.
\end{corollary}
\begin{proof}
Taking $(r_0,p_0)=(1,\infty)$, $(r_1,p_1)=(2,2)$ in theorem \ref{thm.lorentz-inter}, the conditions in \eqref{eqn.interpolate-indices} translate into $\frac{1}{p}=\frac{\theta}{2}$ and $\frac{1}{r}=1-\frac{\theta}{2}$, that is, $r=p'$. Furthermore take $q=r$. Since $\theta\in(0,1)$ in the hypothesis of theorem \ref{eqn.interpolate-indices}, the role of $p$ and $p'$ must be exchanged when we consider the setup in the present corollary. (Since $\frac{2}{p}=\theta\in(0,1)$ if and only if $p>2$). With this adjustment in mind, the conclusion to theorem \ref{eqn.interpolate-indices} becomes
$\|\what{f}\|^*_{p',p}\leq B\|f\|^*_{p,p}=B\|f\|_p$.
\end{proof}

As in the proof of corollary \ref{cor.HY-Lorentz}, we obtain the following extension immediately from the interpolation theorem \ref{thm.lorentz-inter}.
\begin{corollary}
The Heckman--Opdam transform is a continuous mapping from $L^p(A,d\mu)^W$ to $L^{p',p}(i\a^*,d\nu)^W$ whenever $1<p<2$.
\end{corollary}

The preceding two corollaries are stronger than their respective standard forms since $L^{p^\prime,p}$ is continuously and properly embedded in $L^{p^\prime}$.

The last result on Lorentz spaces that we will need is due to R. O'Neil, \cite{Oneil-convolution}, and concerns the pointwise product of two functions. 

\begin{theorem}\label{thm.Oneil} 
	Let $q\in(2,\infty)$ and set $r=\frac{q}{q-2}$. For $g\in L^q(X)$ and $h\in L^{r,\infty}(X)$ it holds that $gh$ belongs to $L^{q',q}(X)$ with $\|gh\|^*_{q',q}\leq\|g\|_q\|h\|_{r,\infty}^*$.
\end{theorem}

\section{The Hardy--Littlewood inequalities}\label{sec.HY-ineqs}
The first part of the present section generalizes the inequality \eqref{eqn.MRSS}. We have decided to treat the rank one case separately as an illustrative example of the interpolation arguments that will be used throughout the section. 
\smallskip

Assume $\mathrm{dim}\,\a=1$, $m_\alpha+m_{2\alpha}\geq 1$, and define $Tf(\lambda)=|\lambda|^2\mathcal{F}f(\lambda)$. Since $-m_\alpha-m_{2\alpha}+1\leq 0$, it follows that $(1+|\lambda|)^{-(m_\alpha+m_{2\alpha})+1}\leq 1$ for all $\lambda\in i\a^*$. Moreover,
\[\begin{split}
\|Tf\|_2^2 &=\int_{i\a^*}|Tf(\lambda)|^2|W|^{-1}(1+|\lambda|)^{-(m_\alpha+m_{2\alpha})+1}|\lambda|^{-4}|\cfct(\lambda)|^{-2}\,d\lambda\\
&= |W|^{-1}\int_{i\a^*}|\mathcal{F}f(\lambda)|^2(1+|\lambda|)^{-(m_\alpha+m_{2\alpha})+1}|\cfct(\lambda)|^{-2}\,d\lambda\\
&\leq |W|^{-1}\int_{i\a^*}|\mathcal{F}f(\lambda)|^2|\cfct(\lambda)|^{-2}\,d\lambda = \|f\|_2^2,
\end{split}\]
so $T$ is of strong type $(2,2)$ as an operator from $L^2(A,d\mu)^W$ into $L^2(i\a^*,d\overline{\nu})^W=L^2(i\a^*,|W|^{-1}(1+|\lambda|)^{-(m_\alpha+m_{2\alpha})+1}|\lambda|^{-4}d\nu(\lambda))^W$. This is no longer true when $m_\alpha+m_{2\alpha}<1$, in which case one would have to employ a different type of weight and/or modify the measure $d\overline{\nu}$.

Note that $|\mathcal{F}f(\lambda)|\leq C\|f\|_1$ for all $\lambda\in i\a^*$ and $f\in L^1(A,d\mu)$.. For $t>0$ and $0\neq f\in L^1(A,d\mu)^W$, define 
\[E_t(f)=\{\lambda\,:\,|Tf(\lambda)|>t\},\quad 
A_t(f)=\Bigl\{\lambda\,:\,|\lambda|>\Bigl(\frac{t}{C\|f\|_1}\Bigr)^{1/2}\Bigr\},\quad\text{and}\quad
a_t =\Bigl(\frac{t}{C\|f\|_1}\Bigr)^{1/2}.\]
It follows from the definition of $T$ that $E_t(f)\subset A_t(f)$ for all $t>0$, hence $|E_t(f)|\leq|A_t(f)|$, where
\[\begin{split}
|A_t(f)| &= \frac{1}{|W|}\int_{A_t(f)} (1+|\lambda|)^{-(m_\alpha+m_{2\alpha})+1}|\lambda|^{-4}|\cfct(\lambda)|^{-2}\,d\lambda\\
&\asymp \frac{1}{|W|}\int_{A_t(f)} (1+|\lambda|)^{-(m_\alpha+m_{2\alpha})+1}|\lambda|^{-4}|\lambda|^2(1+|\lambda|)^{m_\alpha+m_{2\alpha}-2}\,d\lambda\\
&=\frac{1}{|W|}\int_{A_t(f)}|\lambda|^{-2}(1+|\lambda|)^{-1}\,d\lambda=\frac{1}{|W|}\int_{a_t}^\infty |\lambda|^{-2}(1+|\lambda|)^{-1}\,d\lambda\\
&\approx \frac{1}{|W|}\int_{a_t}^\infty\frac{d\lambda}{\lambda^3} = Ca_t^{-2}=C'\frac{\|f\|_1}{t}
\end{split}
\]
This shows that $T$ is also of weak type $(1,1)$ as an operator from $L^2(A,d\mu)^W$ into $L^2(i\a^*,d\overline{\nu})^W$. It follows from the Marcinkiewicz interpolation theorem that $T$ is of strong type $(p,p)$, for $p\in(1,2)$, that is
\begin{equation}\label{eqn.inter-ineq}
\frac{1}{|W|}\int_{i\a^*}|Tf(\lambda)|^p(1+|\lambda|)^{-(m_\alpha+m_{2\alpha})+1}|\lambda|^{-4}\,d\nu(\lambda)\leq C\|f\|_p^p.
\end{equation}
Since $1+|\lambda|\geq|\lambda|^{2-p}$ for all $\lambda$, it holds that $(1+|\lambda|)^{-(m_\alpha+m_{2\alpha})+1}\geq|\lambda|^{2-p}(1+|\lambda|)^{-(m_\alpha+m_{2\alpha})}$, which allows us to rewrite \eqref{eqn.inter-ineq} as 
\[\frac{1}{|W|}\int_{i\a^*}|\mathcal{F}f(\lambda)|^p|\lambda|^{p-2}(1+|\lambda|)^{-(m_\alpha+m_{2\alpha})}\,d\nu(\lambda)\leq C\|f\|_p^p.\]

\begin{remark}
Consider the more natural weighted measure $d\widetilde{\nu}(\lambda)=|W|^{-1}|\lambda|^{-4}|\cfct(\lambda)|^{-2}\,d\lambda$. The operator $T$ is then of strong type $(2,2)$ as an operator from $L^2(A,d\mu)^W$ into $L^2(i\a^*,d\widetilde{\nu})^W$ \emph{for all} $m_\alpha,m_{2\alpha}$ but it is no clear if $T$ is of weak type $(1,1)$. A good choice of weights is therefore essential.

It is also possible to consider weighted measures of the form $d\what{\mu}(\lambda)=\psi(\lambda)|\cfct(\lambda)|^{-2}d\lambda$ where 
\[\psi(\lambda)=\begin{cases}\psi_1(\lambda)&\text{for }\|\lambda\|\leq 1\\ \psi_2(\lambda)&\text{for } \|\lambda\|>1\end{cases}\]
for suitable choices of $\psi_1,\psi_2$, but this leads to so much freedom that one should no longer speak of Hausdorff--Young inequalities. It would, however, allow one to treat the case $0\leq m_\alpha+m_{2\alpha}<1$ as well.
\end{remark}

The next result may be seen as a weighted Hardy--Littlewood inequality, from which a natural analogue of the Hardy--Littlewood inequality in \cite{Mohanty-II} will follow. It is important to allow a certain freedom in the choice of weights, since it would otherwise be difficult to `guess' the correct formulation in higher rank. The parameter constraints arise from having to be able to find an elementary proof of the required weak type $(1,1)$ estimate. Recall that $\beta=\sum_{\alpha\in\Sigma_0^+}(m_\alpha+m_{2\alpha})$, where $m:\Sigma\to[0,\infty)$ is a non-negative multiplicity function, and $m_\alpha:=m(\alpha)$.

\begin{prop}\label{prop.abstract-interpolate}
Assume $\beta+n>0$ and $1<p<2$, and consider the operator $T$ defined on $L^2(X,d\mu)$, $X=A/W$, by
$Tf(\lambda)=\|\lambda\|^{k+n}\mathcal{F}f(\lambda)$, where $k\geq 0$. Moreover let 
\[(Y,d\overline{\nu})=\Bigl(i\a^*, \frac{1}{\vert W\vert}\|\lambda\|^a(1+\|\lambda\|)^b|\cfct(\lambda)|^{-2}\,d\lambda\Bigr)\]
where the parameters $k,a,b$ satisfy the conditions 
\begin{enumerate}[label=(\roman*)]
\item $a+b\leq \frac{2}{3}(n-\beta)$
\item $a+b+\beta+n=-(k+n)$
\end{enumerate}
Then $T$ is of strong type $(2,2)$ and weak type $(1,1)$ as an operator from $L^2(X,d\mu)$ into $L^2(Y,d\overline{\nu})$, and therefore of strong type $(p,p)$. More precisely, there exists a positive constant $C_p$ independent of $f$ such that
\[\frac{1}{|W|}\int_{i\a^*}\vert\mathcal{F}f(\lambda)\vert^p \|\lambda\|^{(k+n)p+a}(1+\|\lambda\|)^b\,d\nu(\lambda)\leq C_p\|f\|_p^p\]
for every $f\in L^p(A,d\mu)^W$.
\end{prop}

\begin{proof}
Consider the measure spaces $(X,d\mu)=(A/W,d\mu)$ and 
\[(Y,d\overline{\nu})=\Bigl(i\a^*, \frac{1}{\vert W\vert}\|\lambda\|^a(1+\|\lambda\|)^b|\cfct(\lambda)|^{-2}\,d\lambda\Bigr).\]
Let $Tf(\lambda)=\|\lambda\|^{k+n}\mathcal{F}f(\lambda)$, where $k$ is to be determined shortly. As in the rank one calculation that preceded the present theorem, the crux of the proof will be to verify that $T$ is strong type $(2,2)$ and weak type $(1,1)$ as an operator from $(X,d\mu)$ into $(Y,d\overline{\nu})$. As for the first property, it follows from Plancherel's theorem that
\[\|Tf\|_2^2 =\frac{1}{|W|}\int_{i\a^*}|\mathcal{F}f(\lambda)|^2\|\lambda\|^{2(k+n)}\|\lambda\|^a(1+\|\lambda\|)^b|\cfct(\lambda)|^{-2}\,d\lambda \simeq \|f\|_2^2\]
\emph{provided} $\psi(\lambda)=\|\lambda\|^{2(k+n)}\|\lambda\|^a(1+\|\lambda\|)^b\asymp 1$. This holds whenever $2(k+n)+a+b\leq 0$ (the factor $\psi$ stays bounded for $\|\lambda\|\gg 1$) and at the same time $2(k+n)+a\geq 0$ (so that $\psi$ is bounded near $\lambda=0$).

For $t>0$ and $f\in L^1\cap L^2$, $\|f\|_1\neq 0$, consider the sets $E_t(f)=\{\lambda\,:\,\vert Tf(\lambda)\vert>t\}$ and $A_t(f)=\{\lambda\,:\,\|\lambda\|>(\frac{t}{C\|f\|_1})^{\frac{1}{k+n}}\}$, where $C$ is the constant coming from the estimate $\vert\mathcal{F}f(\lambda)|\leq C\|f\|_1$. By definition of $T$ it follows that $E_t(f)\subset A_t(f)$, whereby
\[\begin{split}
|E_t(f)|& \leq |A_t(f)| =\frac{1}{|W|}\int_{A_t(f)} \|\lambda\|^a(1+\|\lambda\|)^b|\cfct(\lambda)\|^{-2}\,d\lambda\\ &\simeq \int_{A_t(f)} \|\lambda\|^{a+2\vert\Sigma_0^+\vert}(1+\|\lambda\|)^{b+\beta-2\vert\Sigma_0^+\vert}\,d\lambda\simeq \int_{A_t(f)}\|\lambda\|^{a+b+\beta}\,d\lambda \\ &= C'\int_{a_t}^\infty s^{a+b+\beta}s^{n-1}\,ds=C'\Bigl(\Bigl(\frac{t}{C\|f\|_1}\Bigr)^{\frac{1}{k+n}}\Bigr)^{a+b+\beta+n} = C''\frac{\|f\|_1}{t}
\end{split}\]
since $a+b+\beta+n=-(k+n)$ by construction\footnote{A natural choice would be to choose $k=\beta$, in which case one should add the assumption $\beta+n>0$. This is automatic for symmetric spaces where root multiplicities are integers, but for more general choices of root multiplicities this could be violated.}.
\end{proof}

The extension to the case $p>2$ utilizes the stronger interpolation result theorem \ref{thm.lorentz-inter} and is motivated by the rank one statement in \cite[Theorem~4.5]{Mohanty-II}. A similar argument leads to a generalization of \cite[Lemma~§4.1]{Anker-besov} but we leave it to the interested reader to write down the details.

\begin{definition}
	A \emph{Young function} is a measurable function $\psi:\a_+\to\R$ with the property that $\mu(\{x\in\a\,:\,\vert\psi(x)\vert\leq t\})\lesssim t$ for all $t>0$.
\end{definition}
\begin{example}\label{example.young}
	In $\R^n$, the function $\psi(x)=\|x\|^m$ is a Young function in $\R^n$ if and only if $m=n$, since $|\{x\in\R^n\,:\,\|x\|^m<t\}|=|B(0,t^{1/m})|=Ct^{n/m}$. Since norms on $\R^n$ are equivalent, the $n$.th power of \emph{any} norm on $\R^n$ gives rise to a Young function.
\end{example}
It is easy to construct Young functions associated with $(\a,\Sigma,m)$. An infinite family of examples is given by $\psi(x)=h(x)J(x)$, where $h(x)=h_0(\|x\|)$ is radial and satisfies $\int_0^\infty \frac{s^n-1}{h_0(s)}ds<\infty$. In this case,
\begin{equation}\label{ineq.sublevel}
\mu(\{x\,:\,\vert\psi(x)\vert\leq t\})\leq t\int_{\a}\frac{dx}{h(x)} =Ct\int_0^\infty\frac{s^{n-1}}{h_0(s)}\,ds = C't\end{equation}
for every $t>0$. For many purposes the estimate in \eqref{ineq.sublevel} is too crude, however. The norm power $\|\cdot\|^n$ would not meet the requirement that $\int_0^\infty s^{n-1}/s\,ds$ be finite, for example, so the estimate \eqref{ineq.sublevel} is only sensible when the measure of the sublevel sets $\{x\,:\,|\psi(x)|<t\}$  cannot be estimated directly. One such example is the following.

\begin{definition}
	The \emph{hyperbolic} Young function associated with $(\a,\Sigma,m)$ is the function $\psi_h(x)=\cosh(\|x\|)J(x)$. This is a root system analogue of the Young function considered in \cite[Lemma~4.4]{Mohanty-II}. \end{definition}

Let $\psi$ be a Young function for $(\a,\Sigma,m)$ and $q>2$. The space $L^{(q)}_\psi(\a_+)$ consists of all measurable $W$-invariant functions $f:\a\to\C$ such that
\[\|f\|_{(q),\psi}:=\Bigl(\int_{\a_+}\vert f(x)\vert^q\psi(x)^{q-2}\,d\mu(x)\Bigr)^{1/q}<\infty.\]
In other words, $f$ belongs to $L^{(q)}_\psi$ if and only if $f\cdot\psi^{1-\frac{2}{q}}$ belongs to $L^q$.
\medskip

The next result is a generalization of theorem \ref{thm.1} in the introduction.

\begin{theorem}\label{thm.HY-qlarge}
	Let $q>2$ and $f\in L^{(q)}_{\psi_h}(\a_+)$. There exists a positive constant $D_q$ independent of $f$ such that
	\[\frac{1}{|W|}\int_{i\a^*}|\mathcal{F}f(\lambda)|^q\,d\nu(\lambda)\leq D_q^q\|f\|^q_{(q),\psi}.\]
\end{theorem}
\begin{proof}
	Let $f$ be a simple function on $A$ and let $Tf(\lambda)=\mathcal{F}f(\lambda)$ (we do not need to add weights to the operator that enters the interpolation argument). Then $\|Tf\|^*_{\infty,\infty}=\|Tf\|_\infty\leq C\|f\|_1=\|f\|^*_{1,1}$, and by the Plancherel theorem it furthermore holds that $\|Tf\|_{2,\infty}^*\leq\|Tf\|^*_{2,2}\leq\|f\|_2\leq\|f\|^*_{2,1}$. By interpolation (cf. theorem \ref{thm.lorentz-inter}) it follows that $\|Tf\|^*_{q,q}\leq \|f\|^*_{q',q}$.
	
	Now define $g(x)=f(x)\psi_h(x)^{1-\frac{2}{q}}$, where $\psi_h(x)=\cosh(\|x\|)J(x)$. Then $g$ belongs to $L^q(A,d\mu)^W$ by hypothesis, since
	\[\|g\|^q_q=\int_A\vert f(x)\vert^q\vert\psi_h(x)\vert^{q-2}\,d\mu(x) = \|f\|^q_{(q),\psi_h}\]
	It follows from the sublevel set estimate implied by $\psi_h$ being a Young function for $(\a,\Sigma,m)$ that 
	\[\mu\bigl(\{x\,:\,\vert\psi_h(x)\vert^{\frac{2}{q}-1}>t\}\bigr) = \mu\Bigl(\Bigl\{x\,:\,\vert\psi_h(x)\vert^{1-\frac{2}{q}}<\frac{1}{t}\Bigr\}\Bigr)\leq Ct^{-\frac{q}{q-2}},\]
	whence $\psi^{\frac{2}{q}-1}$ belongs to $L^{r,\infty}(A,d\mu)$, where $r=\frac{q}{q-2}$. By an application of O'Neil's theorem \ref{thm.Oneil} it is seen that
\[\begin{split}
\frac{1}{|W|}\int_{i\a^*}|\mathcal{F}f(\lambda)|^q\,d\nu(\lambda)&\leq \|f\|^*_{p,q}\leq \|g\|_q\|\psi_h\|^*_{t,\infty}\\ & \leq C\int_A|f(x)|^q|\psi_h(x)|^{q-2}J(x)\,dx = c\|f\|^p_{(q),\psi_h}
\end{split}\]
which was the desired conclusion for simple functions. The extension to general functions in $L^{(q)}_{\psi_h}(\a_+)$ now follows by standard density arguments.
\end{proof}

As briefly mentioned in the introduction, Ray and Sarkar were able to obtain a different version of the Hardy--Littlewood inequality for the Helgason--Fourier transform. They might have been motivated by the complex version of the Hausdorff--Young inequality, cf. \eqref{ineq-HY-root-2}. where the transform is extended holomorphically into a certain domain in the complex plane. At the same time Ray and Sarkar used slightly different weights in their interpolation argument, the result being the following theorem, cf. \cite[Theorem~4.11]{Ray-Sarkar-trans} (in their notation $S$ denotes a Damek--Ricci space, but the reader may replace it with a hyperbolic space).

\begin{theorem}\label{thm.Ray-Sarkar}
\begin{enumerate}[label=(\roman*)]
\item Let $1<q\leq 2$ be fixed. Then for $f\in L^p(S)$, $1<p\leq q$,
\[\Bigl(\int_\R\|\widetilde{f}(\lambda\pm i\gamma_q\rho,\cdot)\|_{L^q(N)}^r (|\lambda||\cfct(\lambda)|^{-2})^{r/p'-1}|\cfct(\lambda)|^{-2}\,d\lambda\Bigr)^{1/r}\leq C_{\pm,p,q}\|f\|_p\]
where $\frac{1}{r}=1-\frac{q'-1}{p'}$.
\item Let $2\leq q<\infty$ be fixed. Then for $f\in L^{(p)}(S)$ with $q\leq p<\infty$, 
\[\Bigl(\int_\R\|\widetilde{f}(\lambda\pm i\gamma_{q'}\rho,\cdot)\|^p_{L^{q'}(N)}|\cfct(\lambda)|^{-2}\,d\lambda\Bigr)^{1/p}\leq C_{\pm,p,q}\|f\|_{(p)}.\]
\end{enumerate}
\end{theorem}
Here $\|f\|_{(p)}=\Bigl(\int_S|f(x)|^pJ(x)^{p-2}\,dx\Bigr)^{1/p}$, 
where $J(x)=(\sinh\frac{r(x)}{2})^m(\sinh r(x))^l$ is essentially the Jacobian associated with polar coordinates in $S$. There are subtle technical issues pertaining to the proper domain of definition of the Helgason--Fourier transform that make the proof of theorem \ref{thm.Ray-Sarkar} more involved than what might be expected, but the strategy of proof is still to use interpolation. Indeed the main task is again to identify two suitable measure spaces and a sublinear operator in such a way that the abstract interpolation machinery produces the desired inequality. In statement (i), one considers measure spaces $(S,dx)$ and $(\R^\times,d\overline{\mu}(\lambda))$, where $d\overline{\mu}(\lambda)=|\lambda|^{-q}|\cfct(\lambda)|^{-2(1-q)}\,d\lambda$, and a sublinear operator $T$ defined for $f\in L^1(S)+L^q(S)$ by $Tf(\lambda)=\|\widetilde{f}(\lambda+i\gamma_q\rho,\cdot)\|_{L^q(N)}(|\lambda||\cfct(\lambda)|^{-2})^{q/q'}$. It can be verified that $T$ is of strong type $(q,q')$ (which follows from a Hausdorff--Young inequality and the Plancherel theorem) and weak type $(1,1)$ (which requires more work). An interpolation argument yields the conclusion in (i).

In (ii), it is convenient to consider measure spaces $(S,dx)$ and $(\R,|\cfct(\lambda)|^{-2}d\lambda)$ and a sublinear operator $T$ defined for $f\in L^1(S)+L^{p',1}(S)$ by $Tf(\lambda)=\|\widetilde{f}(\lambda+i\gamma_{q'}\rho,\cdot)\|_{L^{q'}(N)}$. While we shall not give the details of their interpolation argument, it also uses interpolation between Lorentz spaces, showing that for $q\leq<\infty$ and $1\leq s\leq\infty$, it holds that
$\|Tf\|_{p,s}^*\leq C_{p,q}\|f\|^*_{p',s}$, cf. \cite[Eqn. (4.27)]{Ray-Sarkar-trans}. An important difference is that Ray and Sarkar use the function $J$ as Young function, which also dictates their definition of $\|f\|_{(p)}$. Let $u=\frac{p}{p-2}$ and $g(x)=f(x)J(x)^{1/u}$. Then $g\in L^p(S)$ and $\|g\|_p=\|f\|_{(p)}$. Moreover $m(\{x\in S\|:\| J(x)\leq t\})\leq Ct$ for all $t>0$, for some constant $C$, where $m$ is Haar measure on $S$ (so $J$ is indeed a Young function in our terminology). Consequently, $J(x)^{-1/u}\in L^{u,\infty}(S)$. It follows by H\"older's inequality that $\|f\|^*_{p',p}\leq C_p\|g\|_p\|J^{-1/u}\|_{u,\infty}$, and therefore (taking $s=p$)
$\|Tf\|_{p,p}^*\leq C_{p,q}\|g\|_p$, from which (ii) follows.
\medskip

In order to generalize theorem \ref{thm.Ray-Sarkar}, we must therefore choose suitable measure spaces, define convenient sublinear operators $T$ and finally choose a good Young function.

\begin{lemma}\label{lemma-J}
Let $(\a,\Sigma,m)$ be a fixed root datum. The function \[\psi(x)=J(x)=\prod_{\alpha\in\Sigma^+}|e^{\alpha(x)}-e^{-\alpha(x)}|^{m_\alpha}\] is a Young function.
\end{lemma}
\begin{proof}
Since $|e^{\alpha(x)}-e^{-\alpha(x)}|\leq 2e^{\|\alpha\|\|x\|}$ for every $\alpha\in\Sigma^+$, $x\in\a$, it follows that $|J(x)|\leq C e^{2\|\rho\|\|x\|}$ for all $x\in\a$, where $\rho=\frac{1}{2}\sum_{\alpha\in\Sigma^+}m_\alpha\alpha$. Since $\int_0^\infty s^{n-1}e^{-2\|\rho\|s}ds<\infty$ for all $n\in\N$, it follows as in the discussion succeeding example \ref{example.young} that $\psi$ is a Young function.
\end{proof}
The following result is a direct analogue of \cite[Theorem~4.11]{Ray-Sarkar-trans}, and our proof follows theirs closely. The main addition is that we once again have to choose the underlying measure spaces and the operator $T$ in such a way that the weak type $(1,1)$ estimate -- now with respect to a weighted measure in the $n$-dimensional vector space $i\a^*$. As in the  proof of theorem \ref{thm.HY-qlarge} this is essentially achieved by working our way backwards from the desired weak type $(1,1)$ estimate, modifying $T$ accordingly. We shall not work with an arbitrary Young function but rather the choice in lemma \ref{lemma-J}. Accordingly $L^{(q)}(\a)$, $q>2$ denotes the space of measurable $W$-invariant functions $f:\a\to\C$ for which
\[\|f\|_{(q)}:=\Bigl(\int_\a|f(x)|^qJ(x)^{q-2}d\mu(x)\Bigr)^{1/q}<\infty.\]

\begin{theorem}\label{thm.HL-ver3}\begin{enumerate}[label=(\roman*)]
\item Let $1<q\leq 2$ be fixed. For $f\in L^p(\a,d\mu)^W$ with $1<p\leq q$,
\[\Bigl(\int_{i\a^*}|\mathcal{F}f(\lambda+\eta)|^r(\|\lambda\||\cfct(\lambda)|^{-2})^{r/p'-1}\,d\nu(\lambda)\Bigr)^{1/r}\leq C_{p,q,\eta}\|f\|_p\]
for every $\eta$ in the interior of $C(\epsilon_p\rho)$, where $\frac{1}{r}=1-\frac{q'-1}{p'}$.
\item Let $2\leq q<\infty$ be fixed. For $f\in L^{(p)}(\a)$ with $q\leq p<\infty$,
\[\Bigl(\int_{i\a^*}|\mathcal{F}f(\lambda+\eta)|^p\,d\nu(\lambda)\Bigr)^{1/p}\leq C_{p,q,\eta}\|f\|_{(p)}\]
for every $\eta$ in the interior of $C(\epsilon_p\rho)$.
\end{enumerate}
\end{theorem}
Note that the special case $p=q=r=2$ recovers the a special case of the Hausdorff--Young inequality in lemma \ref{HL-ineq2} and the Plancherel formula as an inequality in the case $\eta=0$.

\begin{proof}[Proof of (i)]
Fix $q\in(1,2]$ and consider the measure spaces $(\a_+,d\mu)$ and $(i\a^+,d\overline{\nu}(\lambda)$, where 
$d\overline{\nu}(\lambda)=\|\lambda\|^{-nq}|\cfct(\lambda)|^{-2(1-nq)}\,d\lambda$. Define $Tf$, $f\in L^p(\a,d\mu)^W$, by $Tf(\lambda)=|\mathcal{F}f(\lambda+\eta)|(\|\lambda\||\cfct(\lambda)|^{-2})^{\frac{q\cdot n}{q'}}$. It then follows from lemma \ref{lemma.NPP-2} that
\[
\|Tf\|_{q'}^{q'} =\int_{i\a^*} |Tf(\lambda)|^{q'}\|\lambda\|^{-nq}|\cfct(\lambda)|^{-2(1-nq)}\,d\lambda=
\int_{i\a^*}|\mathcal{F}f(\lambda+\eta)|^{q'}|\cfct(\lambda)|^{-2}\,d\lambda \leq C_{p,q,\eta}\|f\|_q^{q'},\]
so $T$ is of strong type $(q,q')$.

The operator $T$ is furthermore of weak type $(1,1)$, as we shall now prove. The argument is nearly the same as in the proof of theorem \ref{thm.HY-qlarge} but relies on a clever trick employed in the proof of \cite[Theorem~4.11]{Ray-Sarkar-trans}. For $t>0$ define the set $E_t(f)=\{\lambda\in i\a^*\,:\,|Tf(\lambda)>t\}$, that is,
\[E_t(f)=\{\lambda\in i\a^*\,:\,(\|\lambda\||\cfct(\lambda)|^{-2})^{nq/q'}|\mathcal{F}f(\lambda+\eta)|>t\}.\]
According to lemma \ref{lemma.NPP-2}(b), $E_t(f)$ is contained in the set
\[\begin{split}
A_t(f)&:=\{\lambda\in i\a^*\,:\, c_p(\|\lambda\||\cfct(\lambda)|^{-2})^{nq/q'}\|f\|_1>t\}\\
&=\{\lambda\in i\a^*\,:\, \|\lambda\||\cfct(\lambda)|^{-2}>a_t\},\quad a_t=\Bigl(\frac{t}{c_p\|f\|_1}\Bigr)^{\frac{q'}{nq}},
\end{split}\]
We can now invoke the trick of Ray and Sarkar: Noting that 
\[\|\lambda\||\cfct(\lambda)|^{-2}\asymp \|\lambda\|^{2\vert\Sigma_0^+|+1}(1+\|\lambda\|)^{\beta-2|\Sigma_0^+|}=G(\|\lambda\|),\]
where $G(s):=s^{2|\Sigma_0^+|+1}(1+s)^{\beta-2|\Sigma_0^+|}$, it it seen that for $s\geq 0$, $G'(\|\lambda\|)\asymp|\cfct(\lambda)|^{-2}$. It follows that
\[\begin{split}
|A_t(f)|&=\int_{i\a^*}\mathbf{1}_{A_t(f)}(\lambda) (\|\lambda\||\cfct(\lambda)|^{-2})^{-nq}|\cfct(\lambda)|^{-2}\,d\lambda\lesssim \int_{i\a^*} G(\|\lambda\|)^{-q}G'(\|\lambda\|)\,d\lambda\\
&=C\int_{a_t}^\infty s^{-nq}s^{n-1}\,ds\quad\text{ by passage to polar coordinates in }\a\\
&=C'a_t^{n-nq} = C''\Bigl(\frac{t}{\|f\|_1}\Bigr)^{\frac{q'}{nq}(n-nq)}= C'''\frac{\|f\|_1}{t}
\end{split}\]
This proves that $T$ is of weak type $(1,1)$. In particular, it follows by interpolation that $T$ is of strong type $(p,r)$ whenever $p$ satisfies the identity $\frac{1}{p}=\frac{1-t}{1}+\frac{t}{q}$ for some $t\in(0,1)$. With $p$ being given in the hypothesis of (i), this identity holds precisely when $\frac{1}{r}=1-t+\frac{t}{q'}=\frac{p'q-q'}{p'q}$, which establishes the asserted inequality in (i).
\end{proof}

\begin{proof}[Proof of (ii)]
Now consider the measure spaces $(\a/W,d\mu)$ and $(i\a^*/W,d\nu(\lambda))$, together with the operator $Tf(\lambda)=|\mathcal{F}f(\lambda+\eta)|$, where $\eta\in C(\epsilon_p\rho)^\circ$ is fixed.  As in the proof of \cite[Theorem~4.11]{Ray-Sarkar-trans} we shall use lemma \ref{lemma.NPP-2} and $J(x)$ as Young function in a double interpolation argument as follows.

Fix a function $f\in L^{(p)}(\a)$ where $p\geq q\geq 2$ and observe that
\begin{equation}\label{eqn.interpolate1}
\|Tf\|^+_{\infty,\infty}\leq\|f\|^*_{1,1}.
\end{equation}
Indeed, it holds more generally that  $\|Tf\|^*_{\infty,\infty}=\|Tf\|_{\infty}\leq\|f\|_1=\|f\|_{1,1}^*$, which is clear for $\eta=0$ (where it amounts to the Riemann--Lebesgue lemma) and follows in general from the estimate $|\varphi_\lambda(x)|\leq 1$ for $x\in\a$ and $\eta\in C(\epsilon_p\rho)\subset C(\epsilon_1\rho)=C(\rho)$.

In addition
\begin{equation}\label{eqn.intermediate}
\|Tf\|_q\leq C\|f\|_{q'}\quad\text{ for } q\geq 2
\end{equation}
which follows from the Plancherel theorem in the case $q=2$ (in which case $\eta\in C(\epsilon_2\rho)=\{0\}$) and for $q>2$ from lemma \ref{lemma.NPP-2}(a) (since in this case $q'<2$). Consequently
\[\|Tf\|^*_{q,\infty}\leq\|Tf\|^*_{q,q}=\|Tf\|_q\leq C_q\|f\|_{q^\prime}\leq C_q\|f\|^*_{q^\prime,1},\]
that is
\begin{equation}\label{eqn.interpolate2}
\|Tf\|^*_{q,\infty}\leq C_q\|f\|^*_{q^\prime,1}.
\end{equation}

It follows from the interpolation theorem \ref{thm.lorentz-inter} that 
\begin{equation}\label{eqn.interpolate3}
\|Tf\|^*_{p,s}\leq C_{p,q}\|f\|^*_{p^\prime,s}
\end{equation}
for $q\leq p\leq \infty$ and  $1\leq s\leq \infty$. Note that the function $g$ defined by $g(x)=f(x)J(x)^\frac{p-2}{2}$ belongs to $L^p(\a,d\mu)^W$ with norm $\|g\|_{p}=\|f\|_{(p)}$.  Since $J$ is a Young function according to lemma \ref{lemma-J}, it is seen that $\mu(\{x\in\a\,:\,J(x)^{-\frac{p-2}{p}}>t\})\lesssim t^{-\frac{p}{p-2}}$, whence $J^{-\frac{p-2}{p}}$ belongs to the Lorentz space $L^{\frac{p}{p-2},\infty}(\a,d\mu)^W$. By H\"older's inequality for Lorentz spaces, $\|f\|^*_{p,p^\prime}\leq C_p\|g\|_p\|J^{-\frac{p-2}{p}}\|_{\frac{p}{p-2},\infty}$. Combined with \eqref{eqn.interpolate3} in the special case $s=p$, we conclude that
\[\|Tf\|^*_{p,p}\leq C_{p,q}\|g\|_p = C_{p,q}\Bigl(\int_\a|f(x)|^pJ(x)^{p-2}\,d\mu(x)\Bigr)^{1/p}\]
which completes the proof of (ii).
\end{proof}

\section{A remark on the flat analogues}\label{section.flat}
Eguchi and Kumahara also obtained a Hardy--Littlewood inequality for the `flat' spherical transform. Consider the Cartan motion group $G_0=\mathfrak{p}\rtimes K$ and identify the flat Riemannian symmetric space $G_0/K$ with $\p$, which is isomorphic to the tangent space of $G/K$ at the origin $eK$. Spherical analysis on $G_0/K$ was developed by Helgason in \cite{Helgason-dualityIII}, by means of which the Hardy--Littlewood inequality may be stated as follows. For $q\geq 2$ there exists a positive constant $A_{0,q}$ such that
\[\Bigl(\frac{1}{|W}\int_{\a^*}|\widetilde{f}(v)|^qJ(v)\,dv\Bigr)^{1/q}\leq A_{0,q}\Bigl(\int_{G_0}|f(x)|^q\sigma_0(x)^{n(q-2)}\Omega_0(x)^{q-2}\,dx\Bigr)^{1/q}\]
for every $f\in C_c(K\setminus G_0/K)$. Here $\widetilde{f}$ is the generalized Bessel transform (the 'flat' spherical transform) of $f$, $\sigma_0(x)=\|X\|$ for $x=(X,k)\in G_0$,  and $\Omega_0(x)=|W|^{-1}(2\pi)^{n/2}\mathrm{vol}(K/M)J(H)$ for $x=k(H,1)k'$, and $J(H)=\prod_{\alpha\in\Sigma}|\alpha(H)|^{m_\alpha}$. Notice the formal similarity with \eqref{ineq-HL-GK}.

We should like to mention a natural extension of these results to the present root system framework. Contrary to the case of the symmetric space $G/K$ being contracted to the flat space $G_0/K$, the ground space $\a$ remains the same. Instead Ben Sa\"\i d and \O rsted \cite{BSO-Bessel}, and de Jeu \cite{deJeu-PW}, consider a limit transition of the the hypergeometric functions $\varphi_\lambda$, namely the functions $\psi(x)=\lim_{\epsilon\to 0} F_{\lambda/\epsilon}(\epsilon x)$. In the case of rank one symmetric spaces, the function $\psi$ \emph{is} indeed a Bessel function, which is explained as follows. We already know that \[\varphi_\lambda(t)={_2}F_1\Bigl(\frac{i\lambda+\rho}{2},\frac{-i\lambda+\rho}{2};\frac{m_\alpha+m_{2\alpha}+1}{2};-\sinh^2t\Bigr).\]
It can be proved on the basis of the asymptotic estimate
\[\frac{\Gamma(z+a)}{\Gamma(z+b)} = z^{a-b}\Bigl(1+\frac{(a-b)(a+b-1)}{2z}+O(z^{-2})\Bigr)\text{ as } z\to\infty\]
that
\[\psi(\lambda,t)=\Gamma\Bigl(\frac{m_\alpha+m_{2\alpha}+1}{2}\Bigr)\Bigl(\frac{\lambda t}{2}\Bigr)^{-\frac{m_\alpha+m_{2\alpha}-1}{2}}J_{\frac{m_\alpha+m_{2\alpha}-1}{2}}(\lambda t),\]
where $J_\nu$ is the standard Bessel function of the first kind. It is therefore seen that the generalized Bessel functions $\psi$ on the flat symmetric space $G_0/K$ \emph{are} the spherical Bessel functions in the rank one case. At least in the case of $\SO_e(1,n)/\SO(n)$, the associated integral transform is known as the Fourier--Bessel transform (or Hankel transform) in the literature. 

The `flat' Heckman--Opdam transform $\mathcal{F}_0$ is therefore defined as the integral transform arising by integrating a suitable $W$-invariant function on $\a$ against the generalized Bessel function $\psi$. The details can be found in \cite{BSO-Bessel} and \cite[Section~4]{deJeu-PW} but the main point is that the generalized Bessel transform of Ben Sa\"\i d and \O rsted coincides with the symmetric Dunkl transform on $\R^n$, cf. \cite[Theorem~4.15]{deJeu-PW} (it is seen by inspecting the proof of \cite[Theorem~3.15]{BSO-Bessel} that the measures involved in the respective integral transforms coincide as well).

\begin{remark}
It is interesting to note that one can also `contract' the Helgason--Fourier transform on $G/K$ to an integral transform on $G_0/K$. Helgason introduced in \cite{Helgason-flat-horocycle} the so-called \emph{flat horocycle transform} as follows. Let $G/K$ be a Riemannian symmetric space of dimension $n$ and rank $\ell$, and let $X_0=G_0/K\simeq \mathfrak{p}$ denote the tangent space to $G/K$ at the origin $eK$. Let $\Xi_0$ denote the $n$-dimensional manifold consisting of $(n-\ell)$-dimensional affine hyperplanes in $X_0$. The flat horocycle transform is a map that assigns to a function $f$ on $X_0$ the function $\widetilde{f}$ on $\Xi_0$ that is defined by $\widetilde{f}(\xi)=\int_\xi f(Y)\,dm(Y)$, where $dm(Y)$ is the standard Euclidean measure on $\xi$. Its analogue on $G/K$ is then the Helgason--Fourier transform, and it seems likely that the results from \cite[Section~4]{Mohanty-II} have natural analogues for the transform $f\mapsto \widetilde{f}$ acting on functions living on $\p$. 
\end{remark}

It is an advantage of the approach by Ben Sa\"\i d and \O rsted that one also obtains that the relevant measures $d\mu_{0}$ and $d\nu_{0}$ for a Plancherel theorem for the flat transform by means of the limit procedure both coincide with the measure $\omega_m(x)dx$, where $\omega_m(x)=\prod_{\alpha\in\Sigma^+}|\langle\alpha,x\rangle |^{m_\alpha}$ is the standard Plancherel weight for the Dunkl transform $\mathcal{T}_m$.

Since the flat Heckman--Opdam transform $\mathcal{F}_0$ is the symmetrized Dunkl transform, the the following Hardy--Littlewood inequality follows from \cite[Lemma~4.1]{Anker-besov}.
\begin{prop}
If $f\in L^p(A,d\mu_{0})^W$ for some $p\in (1,2)$, then
\[\Bigl(\int_{i\a^*}\|\lambda\|^{2(\rho+\frac{d}{2})(p-2)}\|\mathcal{F}_{0}f(\lambda)|^{p}\,d\nu_{0}(\lambda)\Bigr)^{1/p}\leq C_{0,q}\Bigl(\int_\a |f(x)|^p\,d\mu_{0}(x)\Bigr)^{1/p}.\]
\end{prop}
Although the analogue of the strengthened Hausdorff--Young lemma \ref{HL-ineq2} is false for the Dunkl transform except for $\eta=0$, one can still use the interpolation techniques in the proof of theorem \ref{thm.HL-ver3} to prove the following result, which resembles theorem \ref{thm.HY-qlarge} and is new for the Dunkl transform. Let $L^{(p)}(\R^n,\omega_m)$ denote the space of measurable functions $f$ on $\R^n$ for which
\[\|f\|_{m,(p)}:=\Bigl(\int_{\R^n}|f(x)|^p\omega_{m}(x)^{p-2}\omega_m(x)\,dx\Bigr)^{1/p}<\infty.\]

\begin{prop}\label{prop.Dunkl-RS}
\begin{enumerate}[label=(\roman*)]
\item Let $1<q\leq 2$ be fixed. For $f\in L^p(\R^n,\omega_m)$ with $1<p\leq q$,
\[\Bigl(\int_{\R^n}|\mathcal{T}_mf(x)|^r(\|x\|\omega_m(x))^{r/p'-1}\omega_m(x)\,dx\Bigr)^{1/r}\leq C_{p,q}\|f\|_{L^p(\R^n,\omega_m)}\]
where $\frac{1}{r}=1-\frac{q'-1}{p'}$.
\item Let $2\leq q<\infty$ be fixed. For $f\in L^{(p)}(\R^n,\omega_m)$ with $q\leq p<\infty$,
\[\Bigl(\int_{\R^n}|\mathcal{T}_mf(x)|^p\omega_k(x)\,dx\Bigr)^{1/p}\leq C_{p,q}\|f\|_{m,(p)}.\]
\end{enumerate}
\end{prop}
\begin{proof}[Outline of proof:]
For the first part, with $q\in (1,2]$ fixed, consider the measure spaces $(\R^n,d\mu_m)$ and $(\R^n,d\overline{\mu}_m)$, where $d\mu_m(x)=\omega_m(x)\,dx$ and $d\overline{\mu}_m(x)=\|x\|^{-nq}\omega_k(x)^{1-nq}\,dx$. Define $Tf(x)=|\mathcal{T}_mf(x)|(\|x\|\omega_m(x))^{\frac{nq}{q'}}$. Then $T$ is of strong type $(q,q')$.

The operator $T$ is also of weak type $(1,1)$ but the details are different. One uses that $\|x\|\omega_m(x)\asymp C\|x\|^{2\rho+1}$, instead of the polynomial estimates for $|\cfct(\lambda)|^{-2}$.
\medskip

For the second statement one uses that $\omega_m$ is a Young function on $\R^n$ with respect to the weighted measure $\omega_m(x)dx$.
\end{proof}

In particular, the same inequality holds for $\mathcal{F}_0$ acting on $L^p(\a,d\mu_{0})^W$:
\begin{corollary}
\begin{enumerate}[label=(\roman*)]
\item Let $1<q\leq 2$ be fixed. For $f\in L^p(\a,d\mu_{0})^W$ with $1<p\leq q$,
\[\Bigl(\int_{i\a^*}|\mathcal{F}_0f(\lambda)|^r(\|\lambda\|\omega_m(\lambda))^{r/p'-1}d\nu_{0}(\lambda)\Bigr)^{1/r} \leq C_{p,q}\|f\|_p\]
where $\frac{1}{r}=1-\frac{q'-1}{p'}$.
\item Let $2\leq q<\infty$ be fixed. For $f\in L^{(p)}(\a,d\mu_{0})^W$ with $q\leq p<\infty$,
\[\Bigl(\int_{i\a^*}|\mathcal{F}_0f(\lambda)|^p\omega_m(\lambda)\,d\lambda\Bigr)^{1/p}\leq C_{p,q}\|f\|_{0,(p)},\]
where 
\[\|f\|_{0,(p)}:=\Bigl(\int_{\a}|f(x)|^p\omega_m(x)^{p-2}\,d\mu_{0}(x)\Bigr)^{1/p}.\]
\end{enumerate}
\end{corollary}

\begin{remark}
We obtain a more familiar form of the Hausdorff--Young inequality in proposition \ref{prop.Dunkl-RS} by choosing as Young function a power of the Euclidean norm instead of the density $\omega_m$, that is $\psi(x)=\|x\|^k$ for some $k\geq 0$. As in example \ref{example.young}, $\psi$ is a Young function for a unique choice of $k$. Since
\[
\mu_m(\{x\in\R^n\,:\,\psi(x)\leq t\}) =\mu_m(B(0,t^{1/k})) = \frac{c_m^{-1}}{2^{\gamma+\frac{n}{2}-1}\Gamma(\gamma+\frac{n}{2})}\int_0^{t^{1/k}} r^{2\rho+n-1}\,dr = Ct^{\frac{2\rho+n}{m}}\]
for all $t>0$, it follows that $\psi$ is a Young function if and only if $k=2\rho+n$, which agrees with example \ref{example.young}. The space $L^{(p)}(\R^n,\omega_m)$ now consists of all measurable functions $f:\R^n\to\C$ for which
\[\|f\|_{m,(p)}:=\Bigl(\int_{\R^n}|f(x)|^p\|x\|^{p-2}\omega_m(x)\,dx\Bigr)^{1/p}<\infty.\]
For $f\in L^{(p)}(\R^n,\omega_m)$ with $2\leq p<\infty$ it holds that
\[\Bigl(\int_{\R^n}|\mathcal{T}_mf(x)|^p\omega_m(x)\,dx\Bigr)^{1/p}\leq C_p\Bigl(\int_{\R^n}|f(x)|^p\|x\|^{2(\rho+\frac{n}{2})(p-2)}\omega_m(x)\,dx\Bigr)^{1/p},\]
which is the `dual' form of the Hardy--Littlewood inequality for the Dunkl transform obtained in \cite[Lemma~4.1]{Anker-besov}. A complete extension of \eqref{HL-ineq1} and its dual form \eqref{HL-ineq2} for the Dunkl transform $\mathcal{T}_m$ and the flat Heckman--Opdam transform $\mathcal{F}_0$ has thereby been obtained.
\end{remark}

The Hausdorff--Young and Hardy--Littlewood inequalities for $\mathcal{F}$ and $\mathcal{F}_0$ are formally all but identical. This begs the question: Is it possible to obtain the inequalities for $\mathcal{F}_0$ \emph{directly} from the analogous inequalities for $\mathcal{F}$? To be more precise, the limit transition defining the generalized Bessel functions $\psi$ gives rise to a family of intermediate integral transforms $\mathcal{F}_\epsilon$ that interpolate between $\mathcal{F}$ and $\mathcal{F}_0$. One can establish, say, a Hausdorff--Young inequality for $\mathcal{F}_\epsilon$, $\epsilon\in(0,1]$ that formally interpolates between the Hausdorff--Young inequalities for $\mathcal{F}$ and $\mathcal{F}_0$, respectively, so it is tempting to let $\epsilon$ tend to zero in this $\epsilon$-parametrized Hausdorff--Young inequality and recover the inequality for $\mathcal{F}_0$. In turn this technique would allow one to `generate' a host of new inequalities for the flat transform $\mathcal{F}_0$ from known results for $\mathcal{F}$. There are many technically sound versions of this heuristic principle in classical harmonic analysis, referred to as transference or restriction principles, but it is not yet clear if the techniques can be extended to the setting of root systems. In the case of a Hausdorff--Young inequality for $\mathcal{F}_\epsilon$, for example, one would need  to know that $\limsup_{\epsilon\to 0} C_{\epsilon,p}$ is finite and at the same time take heed of the fact that measures used in the definition of $\mathcal{F}_\epsilon$ also change with $\epsilon$. Although the Hausdorff--Young and Hardy--Littlewood inequalities for $\mathcal{F}_0$ were easy to prove we still find such a philosophy promising and plan to investigate it at length in the near future.

\providecommand{\bysame}{\leavevmode\hbox to3em{\hrulefill}\thinspace}
\providecommand{\MR}{\relax\ifhmode\unskip\space\fi MR }
\providecommand{\MRhref}[2]{%
  \href{http://www.ams.org/mathscinet-getitem?mr=#1}{#2}
}
\providecommand{\href}[2]{#2}

\end{document}